\documentclass{amsart}

\usepackage{macros}
\standardmargins\standardrenews\standardlabeling\standardmakes
\colorcommentstrue

\draftfalse

\RequirePackage{color}
\RequirePackage[usenames,dvipsnames]{xcolor}
\usepackage{tikz}

\usepackage{pgfplots}
\pgfplotsset{compat=1.8}
\usepackage{mathtools}

\usepackage{esint}

\usepackage[hidelinks]{hyperref}

\newcommand{\Sing}{\text{Sing}}
\newcommand{\VSing}{\text{VSing}}

\newcommand{\tbeta}{\gamma}
\newcommand{\tstar}{\gamma_*}

\newcommand{\dimsing}{\delta_\dims}
\newcommand{\dirichlet}{{\tfrac\qdim\pdim}}

\renewcommand{\SS}{\mathcal S}

\renewcommand{\bfA}{A}

\newcommand{\bigOC}[2]{\big(#1,#2\big]}

\newcommand{\firstdim}{m}
\newcommand{\seconddim}{n}
\newcommand{\pdim}{\firstdim}
\newcommand{\qdim}{\seconddim}
\newcommand{\dims}{{\firstdim,\seconddim}}
\newcommand{\dimsum}{{\firstdim + \seconddim}}
\newcommand{\dimprod}{{\firstdim \seconddim}}

\renewcommand{\index}{k}
\newcommand{\mink}{h}
\newcommand{\Mink}{\hh}
\newcommand{\pert}{b}
\newcommand{\Pert}{\bb}
\newcommand{\exceptionalset}{\{\tfrac1\pdim,-\tfrac1\qdim\}}

\newcommand{\timeerror}{s}

\newcommand{\dynexp}{\tau}

\newcommand{\dir}{L}
\newcommand{\diff}{M}
\renewcommand{\vert}{\LL}
\newcommand{\niol}{N}

\begin{document}
\title[A variational principle in the parametric geometry of numbers]{A variational principle in the parametric geometry of numbers, with applications to metric Diophantine approximation}

\authortushar\authorlior\authordavid\authormariusz


\begin{Abstract}
We establish a new connection between metric Diophantine approximation and the parametric geometry of numbers by proving a variational principle facilitating the computation of the Hausdorff and packing dimensions of many sets of interest in Diophantine approximation. In particular, we show that the Hausdorff and packing dimensions of the set of singular $m\times n$ matrices are both equal to $mn \big(1-\frac1{m+n}\big)$, thus proving a conjecture of Kadyrov, Kleinbock, Lindenstrauss, and Margulis as well as answering a question of Bugeaud, Cheung, and Chevallier. Other
applications include computing the dimensions of the sets of points witnessing conjectures of Starkov and Schmidt.\\

{\it R\'esum\'e.} Nous \'etablissons une nouvelle connexion entre l'approximation m\'etrique diophantine et la g\'eom\'etrie param\'etrique des nombres en prouvant un principe variationnel facilitant le calcul des dimensions d'Hausdorff et de packing de nombreux ensembles d'int\'er\^{e}t dans l'approximation diophantienne. Nous montrons que les dimensions pr\'ecit\'ees de l'ensemble des matrices $m\times n$ singuli\`eres sont toutes deux \'egales \`a $mn \big(1-\frac1{m+n}\big)$, d\'emontrant ainsi une conjecture de Kadyrov, Kleinbock, Lindenstrauss, et Margulis et r\'epondant \`a une question de Bugeaud, Cheung, et Chevallier. D'autres applications comprennent le calcul des dimensions des ensembles des points t\'emoignant des conjectures de Starkov et de Schmidt.
\end{Abstract}
\maketitle

\vspace{-0.5 in}
\section{Main results}
\label{sectionmain}

The notion of singularity (in the sense of Diophantine approximation) was introduced by Khintchine, first in 1937 in the setting of simultaneous approximation \cite{Khinchin6}, and later in 1948 in the more general setting of matrix approximation \cite{Khinchin4}.\Footnote{Although Khintchine's 1926 paper \cite{Khinchin3} includes a proof of the existence of $2\times 1$ and $1\times 2$ matrices possessing a certain property which implies that they are singular, it does not include a definition of singularity nor discuss any property equivalent to singularity.} Since then this notion has been studied within Diophantine approximation and allied fields, see Moshchevitin's 2010 survey \cite{Moshchevitin3}. An $\pdim\times \qdim$ matrix $\bfA$ is called \emph{singular} if for all $\epsilon > 0$, there exists $Q_\epsilon$ such that for all $Q \geq Q_\epsilon$, there exist integer vectors $\pp\in \Z^\pdim$ and $\qq\in \Z^\qdim$ such that
\begin{align*}
\|\bfA \qq + \pp\| \leq \epsilon Q^{-\qdim/\pdim}\;\;\;\; \text{ and } \;\;\;\;
0 < \|\qq\| \leq Q.
\end{align*}
Here $\|\cdot\|$ denotes an arbitrary norm on $\R^\pdim$ or $\R^\qdim$. We denote the set of singular $\pdim\times\qdim$ matrices by $\Sing(\pdim,\qdim)$. For $1\times 1$ matrices (i.e. numbers), being singular is equivalent to being rational, and in general any matrix $\bfA$ which satisfies an equation of the form $\bfA \qq = \pp$, with $\pp,\qq$ integral and $\qq$ nonzero, is singular. However, Khintchine proved that there exist singular $2\times 1$ matrices whose entries are linearly independent over $\Q$ \cite[Satz II]{Khinchin3}, and his argument generalizes to the setting of $\pdim\times\qdim$ matrices for all $(\pdim,\qdim) \neq (1,1)$. The name \emph{singular} derives from the fact that $\Sing(\pdim,\qdim)$ is a Lebesgue nullset for all $\pdim,\qdim$, see e.g. \cite[p.431]{Khinchin6} or \cite[Chapter 5, \67]{Cassels}. Note that singularity is a strengthening of the property of \emph{Dirichlet improvability} introduced by Davenport and Schmidt \cite{DavenportSchmidt4}.

In contrast to the measure zero result mentioned above, the computation of the Hausdorff dimension of $\Sing(\pdim,\qdim)$ has been a challenge that so far only met with partial progress. The first breakthrough was made in 2011 by Cheung \cite{Cheung}, who proved that the Hausdorff dimension of $\Sing(2,1)$ is $4/3$; this was extended in 2016 by Cheung and Chevallier \cite{CheungChevallier}, who proved that the Hausdorff dimension of $\Sing(\pdim,1)$ is $\pdim^2/(\pdim+1)$ for all $\pdim\geq 2$; while most recently Kadyrov, Kleinbock, Lindenstrauss, and Margulis \cite{KKLM} proved that the Hausdorff dimension of $\Sing(\pdim,\qdim)$ is at most $\dimsing := \dimprod\big(1 - \tfrac1{\dimsum}\big)$, and went on to conjecture that their upper bound is sharp for all $(\pdim,\qdim)\neq (1,1)$ (see also \cite[Problem 1]{BCC}).

In this paper, we announce a proof that their conjecture is correct. We will also show that the packing dimension of $\Sing(\pdim,\qdim)$ is the same as its Hausdorff dimension, thus answering a question of Bugeaud, Cheung, and Chevallier \cite[Problem 7]{BCC}. To summarize:
\begin{theorem}
\label{theoremsing}
For all $(\pdim,\qdim)\neq (1,1)$, we have
\[
\HD(\Sing(\pdim,\qdim)) = \PD(\Sing(\pdim,\qdim)) = \dimsing \df \dimprod\big(1 - \tfrac1{\dimsum}\big),
\]
where $\HD(S)$ and $\PD(S)$ denote the Hausdorff and packing dimensions of a set $S$, respectively.
\end{theorem}

\subsection{Dani correspondence}
The set of singular matrices is linked to homogeneous dynamics via the \emph{Dani correspondence principle}. For each $t\in\R$ and for each matrix $\bfA$, let
\begin{align*}
g_t &\df \left[\begin{array}{ll}
e^{t/\pdim} \mathrm I_\pdim &\\
& e^{-t/\qdim} \mathrm I_\qdim
\end{array}\right],&
u_\bfA &\df \left[\begin{array}{ll}
\mathrm I_\pdim & \bfA\\
& \mathrm I_\qdim
\end{array}\right],
\end{align*}
where $\mathrm I_k$ denotes the $k$-dimensional identity matrix. Finally, let $d = \pdim + \qdim$, and for each $j = 1,\ldots,d$, let $\lambda_j(\Lambda)$ denote the $j$th successive minimum of a lattice $\Lambda \subset \R^d$ (with respect to some fixed norm on $\R^d$), i.e. the infimum of $\lambda$ such that the set $\{\rr\in\Lambda : \|\rr\| \leq \lambda\}$ contains $j$ linearly independent vectors.. Then the Dani correspondence principle is a dictionary between the Diophantine properties of a matrix $\bfA$ on the one hand, and the dynamical properties of the orbit $(g_t u_\bfA \Z^d)_{t\geq 0}$ on the other. A particular example is the following result:

\begin{theorem}[{\cite[Theorem 2.14]{Dani4}}]
\label{theoremdani}
An $\pdim\times\qdim$ matrix $\bfA$ is singular if and only if the trajectory $(g_t u_\bfA \Z^d)_{t\geq 0}$ is divergent in the space of unimodular lattices in $\R^d$, or equivalently if
\[
\lim_{t\to\infty} \lambda_1(g_t u_\bfA \Z^d) = 0.
\]
\end{theorem}

It is natural to ask about the set of matrices such that the above limit occurs at a prescribed rate, such as the set of matrices such that $-\log\lambda_1(g_t u_\bfA \Z^d)$ grows linearly with respect to $t$. This question is closely linked with the concept of uniform exponents of irrationality. The \emph{uniform exponent of irrationality} of an $\pdim\times \qdim$ matrix $\bfA$, denoted $\what\omega(\bfA)$, is the supremum of $\omega$ such that for all $Q$ sufficiently large, there exist integer vectors $\pp\in \Z^\pdim$ and $\qq\in\Z^\qdim$ such that
\begin{align*}
\|\bfA \qq + \pp\| &\leq Q^{-\omega} \text{ and }
0 < \|\qq\| \leq Q.
\end{align*}
By Dirichlet's theorem, every $\pdim\times\qdim$ matrix $\bfA$ satisfies $\what\omega(\bfA) \geq \dirichlet$. Moreover, it is immediate from the definitions that any matrix $\bfA$ satisfying $\what\omega(\bfA) > \dirichlet$ is singular. We call a matrix \emph{very singular} if it satisfies the inequality $\what\omega(\bfA) > \dirichlet$, in analogy with the set of \emph{very well approximable} matrices, which satisfy a similar inequality for the regular (non-uniform) exponent of irrationality (see \eqref{VWA}). We denote the set of very singular $\pdim\times\qdim$ matrices by $\VSing(\pdim,\qdim)$. The relationship between uniform exponents of irrationality and very singular matrices on the one hand, and homogeneous dynamics on the other, is given as follows:

\begin{theorem}
\label{theoremdani2}
A matrix $\bfA$ is very singular if and only $\what\tau(\bfA) > 0$, where
\[
\what\dynexp(\bfA) \df \liminf_{t\to\infty} \frac{-1}t\log\lambda_1(g_t u_\bfA \Z^d).
\]
Moreover, the quantities $\tau = \what\tau(\bfA)$ and $\omega = \what\omega(\bfA)$ are related by the formula
\begin{equation}
\label{dani}
\tau = \frac1\qdim \frac{\omega - \dirichlet}{\omega + 1}\cdot
\end{equation}
\end{theorem}

This theorem is a straightforward example of the Dani correspondence principle and is probably well-known, but we have not been able to find a reference.

\subsection{Dimensions of very singular matrices}
Perhaps unsurprisingly, the set of very singular matrices has the same dimension properties as the set of singular matrices.

\begin{theorem}
\label{theoremvsing}
For all $(\pdim,\qdim)\neq (1,1)$, we have
\[
\HD(\VSing(m,n)) = \PD(\VSing(m,n)) = \dimsing.
\]
\end{theorem}

One can also ask for more precise results regarding the function $\what\omega$. Specifically, for each $\omega > \dirichlet$ we can consider the levelset\Footnote{For results considering the superlevelset, see Theorem \ref{theoremvariational5}.}
\[
\Sing_\dims(\omega) \df \{\bfA : \what\omega(\bfA) = \omega\} = \{\bfA : \what\tau(\bfA) = \tau\},
\]
where $\tau$ is given by \eqref{dani}. It would be desirable to give precise formulas for the Hausdorff and packing dimensions of $\Sing_\dims(\omega)$ in terms of $\omega$, $\pdim$, and $\qdim$, see e.g. \cite[Problem 2]{BCC}. However, this appears quite challenging at the present juncture, though we have made significant progress towards this question which we will describe in the next section. Thus, instead of precise formulas we will give asymptotic formulas of two types: estimates valid when $\omega$ is small and estimates valid when $\omega$ is large. Note that while the minimum value of $\what\omega$ is always $\dirichlet$ (corresponding to $\what\tau = 0$), the maximum value depends on whether or not $\qdim$ is at least $2$. If $\qdim \geq 2$, then the maximum value of $\what\omega$ is $\infty$ (corresponding to $\what\tau = \frac1\qdim$), while if $\qdim = 1$, then the maximum value of $\what\omega$ (excluding rational points) is $1$ (corresponding to $\what\tau = \frac{\pdim-1}{2\pdim}$). Consequently, we have two different asymptotic estimates of the dimensions of $\Sing_\dims(\omega)$ when $\omega$ is large corresponding to these two cases.

Recall that $\Theta(x)$ denotes any number such that $x/C \leq \Theta(x) \leq C x$ for some uniform constant $C$. In all of the formulas below, $\tau$ is given by \eqref{dani}.

\begin{theorem}
\label{theoremhsmall}
Suppose that $(\pdim,\qdim) \neq (1,1)$. Then for all $\omega > \dirichlet$ sufficiently close to $\dirichlet$, we have\vspace{-0.23in}
\begin{align*}
\phantom{\HD(\Sing_\dims(\omega))} &\hspace{2 in}&
\phantom{\PD(\Sing_\dims(\omega))} &\hspace{1.2 in}\\
\HD(\Sing_\dims(\omega)) &= \dimsing - \Theta\left(\sqrt{\omega - \dirichlet}\right)&
\PD(\Sing_\dims(\omega)) &= \dimsing - \Theta\Big(\omega - \dirichlet\Big)\\
&= \dimsing - \Theta\left(\sqrt{\tau}\,\right)&
&= \dimsing - \Theta\left(\tau\right)
\end{align*}
unless $\pdim = \qdim = 2$, in which case\vspace{-0.23in}
\begin{align*}
\phantom{\HD(\Sing_\dims(\omega))} &\hspace{2 in}&
\phantom{\PD(\Sing_\dims(\omega))} &\hspace{1.2 in}\\
\HD(\Sing_\dims(\omega)) &= \dimsing - \Theta\Big(\omega - \dirichlet\Big)&
\PD(\Sing_\dims(\omega)) &= \dimsing - \Theta\Big(\omega - \dirichlet\Big)\\
&= \dimsing - \Theta\left(\tau\right)&
&= \dimsing - \Theta\left(\tau\right).
\end{align*}
\end{theorem}
\vspace{.05 in}
\begin{theorem}
\label{theoremn2}
Suppose that $\qdim \geq 2$. Then for all $\omega<\infty$ sufficiently large, we have\vspace{-0.23in}
\begin{align*}
\phantom{\HD(\Sing_\dims(\omega))} &\hspace{2 in}&
\phantom{\PD(\Sing_\dims(\omega))} &\hspace{1.2 in}\\
\HD(\Sing_\dims(\omega)) &= \dimprod - 2\pdim + \Theta\left(\tfrac1\omega\right)&
\PD(\Sing_\dims(\omega)) &= \dimprod - \pdim.\\
&= \dimprod - 2\pdim + \Theta\left(\tfrac1\qdim - \tau\right)
\end{align*}
\end{theorem}
\vspace{0 in}
\begin{theorem}
\label{theoremn1}
Suppose that $\qdim = 1$ and $\pdim\geq 2$. Then for all $\omega<1$ sufficiently close to $1$, we have\vspace{-0.23in}
\begin{align*}
\phantom{\HD(\Sing_\dims(\omega))} &\hspace{2 in}&
\phantom{\PD(\Sing_\dims(\omega))} &\hspace{1.2 in}\\
\HD(\Sing_\dims(\omega)) &= \Theta\left(1 - \omega\right)&
\PD(\Sing_\dims(\omega)) &= 1.\\
&= \Theta\Big(\tfrac{\pdim-1}{2\pdim} - \tau\Big)
\end{align*}
\end{theorem}

\begin{remark}
\label{remarktriviallysingular}
Call a matrix $\bfA$ \emph{trivially singular} if there exists $j = 1,\ldots,d-1$ such that
\[
\log\lambda_{j+1}(g_t u_\bfA \Z^d) - \log\lambda_j(g_t u_\bfA \Z^d) \to \infty \text{ as } t\to\infty.
\]
Then all of the above formulas remain true if $\Sing_\dims(\omega)$ is replaced by the set
\[
\Sing_\dims^*(\omega) = \{\bfA\in \Sing_\dims(\omega) : \bfA\text{ is not trivially singular}\}.
\]
Moreover, for $\qdim \geq 2$ we have\vspace{-0.23in}
\begin{align*}
\phantom{\HD(\Sing_\dims(\infty))} &\hspace{1 in}&
\phantom{\PD(\Sing_\dims(\infty))} &\hspace{1 in}\\
\HD(\Sing_\dims^*(\infty)) &= \dimprod - 2\pdim&
\PD(\Sing_\dims^*(\infty)) &= \dimprod - \pdim
\end{align*}
and for $\qdim = 1$, $\pdim \geq 2$ we have\vspace{-0.23in}
\begin{align*}
\phantom{\HD(\Sing_\dims(\infty))} &\hspace{1 in}&
\phantom{\PD(\Sing_\dims(\infty))} &\hspace{1 in}\\
\HD(\Sing_\dims^*(1)) &= 0&
\PD(\Sing_\dims^*(1)) &= 1.
\end{align*}
Note that the class of trivially singular matrices is smaller than the class of matrices with degenerate trajectories in the sense of \cite[Definition 2.8]{Dani4}, but larger than the class considered in \cite[p.2]{BCC} consisting of matrices $\bfA$ such that the group $\bfA\Z^\qdim + \Z^\pdim$ does not have full rank. A $d\times 1$ or $1\times d$ matrix is trivially singular if and only if it is contained in a rational hyperplane of $\R^d$.
\end{remark}
%


\ignore{

{\renewcommand{\arraystretch}{1.8}
\begin{table}[h!]
\begin{tabular}{|c|c|c|c|}
\hline
\spc{Conditions on $(m,n)$}
&
\spc{Conditions on $\omega$}
&
\spc{$\HD(\Sing_\dims(\omega))$}
&
\spc{$\PD(\Sing_\dims(\omega))$}
\\ \hline
\hline
\spc{$(\pdim,\qdim) \notin \{(1,1), (2,2)\}$}
&
\spc{$\omega > \dirichlet$ \\ sufficiently close to $\dirichlet$}
&
\spc{$\dimsing - \Theta\left(\sqrt{\omega - \dirichlet}\right)$\\
$= \dimsing - \Theta\left(\sqrt{\tau}\,\right)$
}
&
\spc{$\dimsing - \Theta\Big(\omega - \dirichlet\Big)$\\
$= \dimsing - \Theta\left(\tau\right)$
}
\\ \hline
$(\pdim,\qdim) = (2,2)$
&
\spc{$\omega > \dirichlet$ \\ sufficiently close to $\dirichlet$}
&
\spc{$\dimsing - \Theta\Big(\omega - \dirichlet\Big)$ \\
$= \dimsing - \Theta\left(\tau\right)$}
&
\spc{$\dimsing - \Theta\Big(\omega - \dirichlet\Big)$ \\
$= \dimsing - \Theta\left(\tau\right)$\\}
\\ \hline
\spc{$\qdim \geq 2$}
&
\spc{$\omega<\infty$ \\ sufficiently large}
&
\spc{$\dimprod - 2\pdim + \Theta\left(\tfrac1\omega\right)$\\
$= \dimprod - 2\pdim + \Theta\left(\tfrac1\qdim - \tau\right)$}
&
\spc{$\dimprod - \pdim$\\ }
\\ \hline
\spc{$\qdim = 1$ and $\pdim\geq 2$\\}
&
\spc{$\omega<1$ \\ sufficiently close to $1$}
&
\spc{$\Theta\left(1 - \omega\right)$\\
$= \Theta\Big(\tfrac{\pdim-1}{2\pdim} - \tau\Big)$\\}
&
$1$
\\ \hline
\end{tabular}
\caption{Hello world.}
\label{}
\end{table}
}

{\renewcommand{\arraystretch}{1.8}
\begin{table}[h!]
\begin{tabular}{|c|c|c|c|}
\hline
\spc{$(m,n)$}
&
\spc{$\omega$}
&
\spc{$\HD(\Sing_\dims(\omega))$}
&
\spc{$\PD(\Sing_\dims(\omega))$}
\\ \hline
\hline
\spc{$(\pdim,\qdim) \notin \{(1,1), (2,2)\}$}
&
$\omega > \dirichlet$ and sufficiently close to $\dirichlet$
&
\spc{$\dimsing - \Theta\left(\sqrt{\omega - \dirichlet}\right)$\\
}
&
\spc{$\dimsing - \Theta\Big(\omega - \dirichlet\Big)$\\
}
\\ \hline
$(\pdim,\qdim) = (2,2)$
&
$\omega > \dirichlet$ and sufficiently close to $\dirichlet$
&
\spc{$\dimsing - \Theta\Big(\omega - \dirichlet\Big)$ \\
}
&
\spc{$\dimsing - \Theta\Big(\omega - \dirichlet\Big)$ \\
}
\\ \hline
\spc{$\qdim \geq 2$}
&
\spc{$\omega<\infty$ and sufficiently large \\ }
&
\spc{$\dimprod - 2\pdim + \Theta\left(\tfrac1\omega\right)$\\
}
&
\spc{$\dimprod - \pdim$\\ }
\\ \hline
\spc{$\pdim\geq 2$ and $\qdim = 1$\\}
&
\spc{$\omega<1$ sufficiently close to $1$\\}
&
\spc{$\Theta\left(1 - \omega\right)$\\
}
&
$1$
\\ \hline
\end{tabular}
\caption{Hello world.}
\label{}
\end{table}
}

{\renewcommand{\arraystretch}{1.8}
\begin{table}[h!]
\begin{tabular}{|c|c|c|c|}
\hline
\spc{$(m,n)$}
&
\spc{$\omega$}
&
\spc{$\HD(\Sing_\dims(\omega))$}
&
\spc{$\PD(\Sing_\dims(\omega))$}
\\ \hline
\hline
\spc{$(\pdim,\qdim) \notin \{(1,1), (2,2)\}$}
&
$\omega > \dirichlet$ and sufficiently close to $\dirichlet$
&
\spc{
$\dimsing - \Theta\left(\sqrt{\tau}\,\right)$
}
&
\spc{
$\dimsing - \Theta\left(\tau\right)$
}
\\ \hline
$(\pdim,\qdim) = (2,2)$
&
$\omega > \dirichlet$ and sufficiently close to $\dirichlet$
&
\spc{
$\dimsing - \Theta\left(\tau\right)$
}
&
\spc{
$\dimsing - \Theta\left(\tau\right)$\\
}
\\ \hline
\spc{$\qdim \geq 2$}
&
\spc{$\omega<\infty$ and sufficiently large \\ }
&
\spc{
$\dimprod - 2\pdim + \Theta\left(\tfrac1\qdim - \tau\right)$
}
&
\spc{$\dimprod - \pdim$\\ }
\\ \hline
\spc{$\pdim\geq 2$ and $\qdim = 1$\\}
&
\spc{$\omega<1$ sufficiently close to $1$\\}
&
\spc{
$\Theta\Big(\tfrac{\pdim-1}{2\pdim} - \tau\Big)$\\
}
&
$1$
\\ \hline
\end{tabular}
\caption{Hello world.}
\label{}
\end{table}
}

}

\subsection{$1\times 2$ and $2\times 1$ matrices}
Although we cannot give precise formulas for the Hausdorff and packing dimensions of $\Sing_\dims(\omega)$ for all pairs $(\pdim,\qdim)$, the special cases $(\pdim,\qdim) = (1,2)$ and $(\pdim,\qdim) = (2,1)$ are easier to handle.

\begin{theorem}
\label{theoremspecialcase}
For all $\omega \in (1/2,1)$ we have
\begin{align*}
\HD(\Sing_{1,2}(\omega)) &= \begin{cases}
\frac43 - \frac43\sqrt{\tau - 6\tau^3 + 4\tau^4} - 2\tau + \frac83 \tau^2
& \text{ if }\tau \leq \tau_0 \df \frac{3\sqrt2-2}{14}\\
\frac{1-2\tau}{1+\tau}
& \text{ if }\tau \geq \tau_0
\end{cases}\\
\PD(\Sing_{1,2}(\omega)) &= \begin{cases}
\tfrac{4-8\tau}{3} & \text{ if } \tau \leq \tau_1 \df \frac18\\
1 & \text{ if } \tau \geq \tau_1
\end{cases}
\end{align*}
(cf. Figure \ref{figurespecialcasegraph}).
\end{theorem}

\begin{figure}
\begin{tikzpicture}
 \begin{axis}[
   xmin=0,
   ymin=0,
   ]
\addplot[domain=0.5858:1.3333, thick, smooth] ((6*(3*x^3-12*x^2+14*x-4)^(1/2) + 9*x-16)/(2*(12*x-25)),x);
\addplot[domain=0.16:0.5, thick,smooth] {2*(1/2-x)/(1+x)};
\addplot[domain=0.0:0.125, thick,smooth] {(4-8*x)/3};
\addplot[domain=0.125:0.5, thick,smooth] {1};
\addplot[mark=*] coordinates {(0.159,0.586)};
\addplot[mark=*] coordinates {(0.125,1.0)};
\end{axis}

\end{tikzpicture}
\caption{The functions $f_0(\tau) =  \HD(\Sing_{1,2}(\omega))$ and $f_1(\tau) = \PD(\Sing_{1,2}(\omega))$. The function $f_0$ is real-analytic on the intervals $[0,\tau_0]$ and $[\tau_0,1/2]$, where $\tau_0 = \frac{3\sqrt2-2}{14} \sim 0.1602$, while $f_1$ is linear on the intervals $[0,1/8]$ and $[1/8,1/2]$.}
\label{figurespecialcasegraph}
\end{figure}
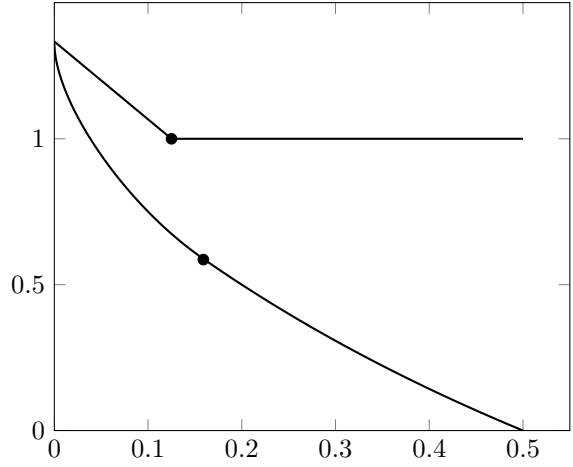

\begin{remark*}
By Jarn\'ik's identity \cite{Jarnik4}, for all $\omega\in [2,\infty]$ we have
\[
\Sing_{1,2}(\omega) = \Sing_{2,1}(\omega')
\]
where $\omega' = 1 - \frac1\omega$. Thus by applying an appropriate substitution to the above formulas, it is possible to get explicit formulas for $\HD(\Sing_{2,1}(\omega'))$ and $\PD(\Sing_{2,1}(\omega'))$, either in terms of $\omega'$ or in terms of $\tau' = \frac{\omega'-\frac12}{\omega'+1} = \frac{\tau}{1+2\tau}$. However, the resulting formulas are not very elegant so we omit them.
\end{remark*}

\begin{remark*}
The transition point $\tau_0 = \frac{3\sqrt 2 - 2}{14}$ in the above formula for Hausdorff dimension corresponds to $\omega_0 = 2+\sqrt2$, $\omega'_0 = \frac{\sqrt2}{2}$, $\tau'_0 = \frac{4-3\sqrt2}{2}$, and $\HD(\Sing_{1,2}(\omega_0)) = 2-\sqrt 2$. The transition point $\tau_1 = \frac18$ for packing dimension corresponds to $\omega_1 = 3$, $\omega'_1 = \frac23$, $\tau'_1 = \frac1{10}$, and $\PD(\Sing_{1,2}(\omega_1)) = 1$.
\end{remark*}

\begin{remark*}
Theorem \ref{theoremspecialcase} implies that $\HD(\Sing_{1,2}(\omega)) < \PD(\Sing_{1,2}(\omega))$ for all $\omega \in (1/2,1)$. This answers the first part of \cite[Problem 7]{BCC} in the affirmative.
\end{remark*}

\begin{remark*}
There has been a lot of partial progress towards the Hausdorff dimension part of Theorem \ref{theoremspecialcase}. In particular, the $\geq$ direction follows from \cite[Corollary 2 and Theorem 3]{BCC}. For $\tau \geq \tau_0$ the upper bound follows from \cite[Corollary 2]{BCC} and for $\tau < \tau_0$, a non-optimal upper bound is given in \cite[Theorem 1]{BCC}. We refer to \cite{BCC} for a detailed history of the prior results.
\end{remark*}

\subsection{Singularity on average}
A different way of quantifying the notion of singularity is the notion of \emph{singularity on average} introduced in \cite{KKLM}. Given a matrix $\bfA$, we define the \emph{proportion of time spent in the cusp} to be the number
\[
\PP(\bfA) = \lim_{\epsilon\to 0} \liminf_{T\to\infty} \frac1T \lambda\big(\big\{t\in [0,T] : \lambda_1(g_t u_\bfA \Z^d) \leq \epsilon\big\}\big) \in [0,1],
\]
where $\lambda$ denotes Lebesgue measure. The matrix $\bfA$ is said to be \emph{singular on average} if $\PP(\bfA) = 1$. Clearly, every singular matrix is singular on average.

\begin{theorem}
For all $p\in [0,1]$, we have
\[
\HD(\{\bfA : \PP(\bfA) = p\}) = 
\PD(\{\bfA : \PP(\bfA) = p\}) = 
p \dimsing + (1-p)\dimprod.
\]
In particular, the dimension of the set of matrices singular on average is $\dimsing$.
\end{theorem}

Note that the fact that the Hausdorff dimension of the set of matrices singular on average is $\leq\dimsing$ was proven in \cite{KKLM}, while the fact that this number is $\geq\dimsing$ follows from Theorem \ref{theoremsing}.

\subsection{Starkov's conjecture}
In \cite[p.213]{Starkov2}, Starkov asked whether there exists a singular vector (i.e. $\pdim\times 1$ singular matrix) which is not very well approximable. Here, we recall that a matrix $\bfA$ is called \emph{very well approximable} if for some $\omega > \dirichlet$, there exist infinitely many pairs $(\pp,\qq)\in\Z^\pdim\times\Z^\qdim$ such that
\begin{equation}
\label{VWA}
\|\bfA \qq + \pp\| \leq \|\qq\|^{-\omega},
\end{equation}
or equivalently in terms of the Dani correspondence principle, a matrix $\bfA$ is very well approximable if $\limsup_{t\to\infty} -\frac1t\log\lambda_1(g_t u_\bfA \Z^d) > 0$. This question was answered affirmatively by Cheung \cite[Theorem 1.4]{Cheung} in the case $\pdim = 2$. In fact, Cheung showed that if $\psi$ is any function such that $q^{1/2}\psi(q) \to 0$ as $q\to\infty$, then there exists a $2\times 1$ singular vector which is not $\psi$-approximable. Here, a matrix $\bfA$ is called \emph{$\psi$-approximable} if there exist infinitely many pairs $(\pp,\qq)\in\Z^\pdim\times\Z^\qdim$ such that $\qq\neq \0$ and
\[
\|\bfA \qq + \pp\| \leq \psi(\|\qq\|).
\]
The following theorem improves on Cheung's result both by generalizing it to the case of arbitrary $\pdim,\qdim$, and also by computing the dimension of the set of matrices with the given property:
\begin{theorem}
\label{theoremstarkov}
If $\psi$ is any function such that $q^{\qdim/\pdim} \psi(q) \to 0$ as $q\to\infty$, then the set of $\pdim\times\qdim$ singular matrices that are not $\psi$-approximable has Hausdorff dimension $\dimsing$. Equivalently, if $\phi$ is any function such that $\phi(t) \to \infty$ as $t \to \infty$, then the set of $\pdim\times\qdim$ singular matrices $\bfA$ such that $-\log\lambda_1(g_t u_\bfA \Z^d) \leq \phi(t)$ for all $t$ sufficiently large has Hausdorff dimension $\dimsing$. The same is true for the packing dimension.
\end{theorem}
Note that this theorem is optimal in the sense that if $\psi(q) \geq c q^{-\qdim/\pdim}$ for some constant $c$, then it is easy to check that every singular $\pdim\times\qdim$ matrix is $\psi$-approximable.

\subsection{Schmidt's conjecture}
In \cite[p.273]{Schmidt7}, Schmidt conjectured that for all $2\leq k \leq \pdim$, there exists an $\pdim\times 1$ matrix $\bfA$ such that
\begin{align}
\label{ksingular}
\lambda_{k-1}(g_t u_\bfA \Z^d) \to 0 \text{ and }
\lambda_{k+1}(g_t u_\bfA \Z^d) \to \infty \text{ as } t\to\infty.
\end{align}
(Note that any matrix satisfying \eqref{ksingular} is singular by Theorem \ref{theoremdani}.) This conjecture was proven by Moshchevitin \cite{Moshchevitin4}, who constructed an $\pdim\times 1$ matrix $\bfA$ satisfying \eqref{ksingular} and not contained in any rational hyperplane of $\R^\pdim$ (see also \cite{Keita,Schleischitz3}). We will improve Moshchevitin's result by computing the Hausdorff and packing dimensions of the set of matrices witnessing this conjecture:

\begin{theorem}
\label{theoremkmessenger}
For all $(\pdim,\qdim)\neq (1,1)$ and for all $2\leq k \leq \dimsum - 1$, the Hausdorff and packing dimensions of the set of matrices $\bfA$ that satisfy \eqref{ksingular} are both equal to
\[
\max(f_\dims(k),f_\dims(k-1))
\]
where
\begin{equation}
\label{fkdef}
f_\dims(k) = \dimprod - \frac{k\dimprod}{\dimsum}\left(1 - \frac k\dimsum\right) - \left\{\frac{k\pdim}{\dimsum}\right\} \left\{\frac{k\qdim}{\dimsum}\right\}.
\end{equation}
Here $\{x\}$ denotes the fractional part of a real number $x$. The same formulas are valid for the set of matrices $\bfA$ that satisfy \eqref{ksingular} and are not trivially singular.
\end{theorem}

\begin{remark}
The function $f_\dims$ satisfies $f_\dims(\dimsum-k)=f_\dims(k)$ and $f_\dims(1) = f_\dims(\dimsum-1) = \dimsing$. Moreover, for all $1\leq k \leq \dimsum-1$ we have $f_\dims(k) \leq \dimsing$. It follows that when $k = 2$ or $\dimsum - 1$, the Hausdorff and packing dimensions of the set of matrices $\bfA$ that satisfy \eqref{ksingular} are both equal to $\dimsing$.
\end{remark}

\begin{remark}
When $\pdim = 1$ or $\qdim = 1$, the fractional parts appearing in \eqref{fkdef} can be computed explicitly, leading to the formula
\[
f_\dims(k) = \dimprod - \frac{k(\dimsum-k)}{\dimsum}\cdot
\]
However, this formula is not valid when $\pdim,\qdim \geq 2$.
\end{remark}~\\

\comdavid{Some other things probably worth adding at some point are:
\begin{itemize}
\item Ahlfors dimension of singular vectors,
\item an upper bound on the dimension of singular vectors in $C\times C$
\item Hausdorff $<$ packing dimension except in special cases
\end{itemize}
However I don't think these should appear in the announcement.
}

{\bf Acknowledgements.} This research began when the authors met at the American Institute of Mathematics via their SQuaRE program. We thank them for their hospitality. The first-named author was supported in part by a 2017-2018 Faculty Research Grant from the University of Wisconsin--La Crosse. The second-named author was supported in part by the Simons Foundation grant \#245708. The third-named author was supported in part by the EPSRC Programme Grant EP/J018260/1. The fourth-named author was supported in part by the NSF grant DMS-1361677. We thank Nicolas Chevallier, Elon Lindenstrauss, Damien Roy, and Johannes Schleischitz for helpful comments.

\section{The variational principle}
\label{sectionvariational}

All the theorems in the previous section (with the exception of Theorems \ref{theoremdani} and \ref{theoremdani2}) are consequences of a single \emph{variational principle} in the parametric geometry of numbers. This variational principle is a quantitative analogue of theorems due to Schmidt and Summerer \cite[\62]{SchmidtSummerer3} and Roy \cite[Theorem 1.3]{Roy3}. However, we will state their results in language somewhat different from the language used in their papers, due to the fact that the fundamental object we consider is the one-parameter family of unimodular lattices $(g_t u_\bfA \Z^d)_{t\geq 0}$ used by the Dani correspondence principle, rather than a one-parameter family of (non-unimodular) convex bodies as is done in \cite{SchmidtSummerer3,Roy3}. We leave it to the reader to verify that the theorems we attribute below to \cite{SchmidtSummerer3} and \cite{Roy3} are indeed faithful translations of their results to our setting. We note that these papers, unlike ours, do not consider the case of matrices.

The fundamental question of our version of the parametric geometry of numbers will be as follows: given a matrix $\bfA$, what does the function $\Mink = \Mink_\bfA = (\mink_1,\ldots,\mink_d) : \Rplus \to \R^d$ defined by the formula
\begin{equation}
\label{hitdef}
\mink_i(t) \df \log\lambda_i(g_t u_\bfA \Z^d)
\end{equation}
look like? The function $\Mink_\bfA$ will be called the \emph{successive minima function} of the matrix $\bfA$. The Dani correspondence principle shows that many interesting Diophantine questions about the matrix $\bfA$ are equivalent to questions about its successive minima function.

The main restriction on the successive minima function comes from an application of Minkowski's second theorem on successive minima to certain subgroups of $g_t u_\bfA \Z^d$. Specifically, fix $j = 1,\ldots,d-1$ and let $I$ be an interval such that $\mink_j(t) < \mink_{j+1}(t)$ for all $t\in I$. For each $t\in I$, let $V_t \subset \R^d$ be the linear span of the set
\[
\{\rr\in \Z^d : \|g_t u_\bfA \rr\| \leq \lambda_j(g_t u_\bfA \Z^d)\}.
\]
Then a continuity argument shows that the map $t\mapsto V_t$ is constant on $I$, see \cite[Lemma 2.1]{SchmidtSummerer1} for the case of simultaneous approximation.
Write $V_t = V$. By Minkowski's second theorem, we have
\[
\sum_{i\leq j} \mink_i(t) \asymp_\plus F_{j,I}(t) \df \log\|g_t u_\bfA (V\cap \Z^d)\|,
\]
where $\|\Gamma\|$ denotes the covolume of a discrete group $\Gamma \subset \R^d$, and $A\asymp_\plus B$ means that there exists a constant $C$ such that $|B-A|\leq C$. Now an argument based on the exterior product formula for covolume and the definition of $g_t$ shows that $F_{j,I} \asymp_\plus G_{j,I}$ for some convex, piecewise linear function $G_{j,I}$ whose slopes are in the set
\begin{equation}
\label{slopeset1}
Z(j) \df \left\{\tfrac{\dir_+}{\pdim} - \tfrac{\dir_-}{\qdim}: \dir_\pm \in [0,d_\pm]_\Z,\;\;\dir_+ + \dir_- = j\right\},
\end{equation}
where for convenience we write $d_+ = \pdim$, $d_- = \qdim$, and $[a,b]_\Z = [a,b]\cap\Z$. This suggests that $\Mink$ can be approximated by a piecewise linear function $\ff$ such that whenever $f_j < f_{j+1}$ on an interval $I$, the function $F_j = \sum_{i\leq j} f_i$ is convex on $I$ with slopes in $Z(j)$. Moreover, it is obvious that $\mink_1\leq \cdots \leq \mink_d$, and the formula for $g_t$ implies that for all $i$, we have $-\frac1\qdim \leq \mink_i' \leq \frac1\pdim$ wherever $\mink_i$ is differentiable. We therefore make the following definition:

\begin{definition}
\label{definitiontemplate}
An $\pdim\times\qdim$ \emph{template} is a piecewise linear map $\ff:\CO{t_0}\infty\to\R^d$ with the following properties:
\begin{itemize}
\item[(I)] $f_1 \leq \cdots \leq f_d$.
\item[(II)] $-\frac1\qdim \leq f_i' \leq \frac1\pdim$ for all $i$.
\item[(III)] For all $j = 1,\ldots,d$ and for every interval $I$ such that $f_j < f_{j+1}$ on $I$, the function $F_j := \sum_{i\leq j} f_i$ is convex on $I$ with slopes in $Z(j)$. Here we use the convention that $f_0 = -\infty$ and $f_{d+1} = +\infty$.
\end{itemize}
When $\pdim = 1$, templates are a slight generalization of reparameterized versions of the \emph{rigid systems} of \cite{Roy3}. We denote the space of $\pdim\times\qdim$ templates by $\TT_{\dims}$.
\end{definition}

The fundamental relation between templates and successive minima functions is given as follows:

\begin{theorem}
\label{theoremSSR}
~
\begin{itemize}
\item[(i)] For every $\pdim\times \qdim$ matrix $\bfA$, there exists an $\pdim\times \qdim$ template $\ff$ such that $\Mink_\bfA \asymp_\plus \ff$.
\item[(ii)] For every $\pdim\times \qdim$ template $\ff$, there exists an $\pdim\times \qdim$ matrix $\bfA$ such that $\Mink_\bfA \asymp_\plus \ff$.
\end{itemize}
\end{theorem}

In the case $\pdim = 1$, part (i) of Theorem \ref{theoremSSR} is due to Schmidt and Summerer \cite[\62]{SchmidtSummerer3} and part (ii) is due to Roy (\cite[Theorem 1.3]{Roy3} and \cite[Corollary 4.7]{Roy2}).

Theorem \ref{theoremSSR}(ii) asserts that for every template $\ff$, the set
\[
\MM(\ff) \df \{\bfA : \Mink_\bfA \asymp_\plus \ff\}
\]
is nonempty. It is natural to ask how big this set is in terms of Hausdorff and packing dimension. Moreover, given a collection of templates $\FF$, we can ask the same question about the set
\[
\MM(\FF) = \bigcup_{\ff\in\FF} \MM(\ff).
\]
It turns out to be easier to answer the second question than the first, assuming that the collection of templates $\FF$ is closed under finite perturbations. Here, $\FF$ is said to be \emph{closed under finite perturbations} if whenever $\gg \asymp_\plus \ff\in\FF$, we have $\gg\in \FF$.

\begin{theorem}[Variational principle, version 1]
\label{theoremvariational1}
Let $\FF$ be a collection of templates closed under finite perturbations. Then
\begin{align}
\label{variational1}
\HD(\MM(\FF)) &= \sup_{\ff\in\FF} \underline\delta(\ff),&
\PD(\MM(\FF)) &= \sup_{\ff\in\FF} \overline\delta(\ff)
\end{align}
where the functions $\underline\delta,\overline\delta:\TT_\dims\to [0,\dimprod]$ are as in Definition \ref{definitiondimtemplate} below.
\end{theorem}

\begin{corollary}
\label{corollaryvariational}
With $\FF$ as above, we have
\begin{align}
\label{variational}
\HD(\MM(\FF)) &= \sup_{\ff\in \FF} \HD(\MM(\ff)),&
\PD(\MM(\FF)) &= \sup_{\ff\in \FF} \PD(\MM(\ff)).
\end{align}
\end{corollary}

However, note that Theorem \ref{theoremvariational1} does not imply that $\HD(\MM(\ff)) = \underline\delta(\ff)$ for an individual template $\ff$, since the family $\{\ff\}$ is not closed under finite perturbations. And indeed, since the function $\underline\delta$ is sensitive to finite perturbations, the formula $\HD(\MM(\ff)) = \underline\delta(\ff)$ cannot hold for all $\ff\in\TT_\dims$.

\begin{definition}
\label{definitiondimtemplate}
We define the lower and upper average contraction rate of a template $\ff$ as follows. Let $I$ be an open interval on which $\ff$ is linear. For each $q = 1,\ldots,d$ such that $f_q < f_{q + 1}$ on $I$, let $\dir_\pm = \dir_\pm(\ff,I,q) \in [0,d_\pm]_\Z$ be chosen to satisfy $\dir_+ + \dir_- = q$ and
\begin{equation}
\label{Lqdef}
F_q' = \sum_{i=1}^q f_i' = \frac{\dir_+}{\pdim} - \frac{\dir_-}{\qdim} \text{ on } I,
\end{equation}
as guaranteed by (III) of Definition \ref{definitiontemplate}. An \emph{interval of equality} for $\ff$ on $I$ is an interval $\OC pq_\Z$, where $0 \leq p < q \leq d$ satisfy
\begin{equation}
\label{pqdef}
f_p < f_{p+1} = \cdots = f_q < f_{q+1} \text{ on } I.
\end{equation}
Note that the collection of intervals of equality forms a partition of $[1,d]_\Z$. If $\OC pq_\Z$ is an interval of equality for $\ff$ on $I$, then we let $\diff_\pm(p,q) = \diff_\pm(\ff,I,p,q)$, where
\begin{equation}
\label{Mpqdef}
\diff_\pm(\ff,I,p,q) = \dir_\pm(\ff,I,q) - \dir_\pm(\ff,I,p),
\end{equation}
and we let
\begin{align} \label{Splusdef1}
S_+(\ff,I) &= \bigcup_{\OC pq_\Z} \bigOC p{p+\diff_+(p,q)}_\Z\\ \label{Sminusdef1}
S_-(\ff,I) &= \bigcup_{\OC pq_\Z} \bigOC{p+\diff_+(p,q)}q_\Z
\end{align}
where the unions are taken over all intervals of equality for $\ff$ on $I$. 
Note that $\diff_\pm(p,q) \geq 0$ by (II) of Definition \ref{definitiontemplate}, and further that $S_+$ and $S_-$ are disjoint and satisfy $S_+\cup S_- = [1,d]_\Z$, and that $\#(S_+) = \pdim$ and $\#(S_-) = \qdim$.
Next, let
\begin{equation}
\label{dimfI}
\delta(\ff,I) = \#\{(i_+,i_-)\in S_+\times S_- : i_+ < i_-\} \in [0,\dimprod]_\Z,
\end{equation}
and note that
\begin{equation}
\label{codimfI}
\dimprod - \delta(\ff,I) = \#\{(i_+,i_-)\in S_+\times S_- : i_+ > i_-\}.
\end{equation}
The \emph{lower and upper average contraction rates} of $\ff$ are the numbers
\begin{align}
\label{deltaFH}
\underline\delta(\ff) &\df \liminf_{T\to\infty} \Delta(\ff,T),&
\overline\delta(\ff) &\df \limsup_{T\to\infty} \Delta(\ff,T),
\end{align}
where
\[
\Delta(\ff,T) \df \frac1T \int_0^T \delta(\ff,t) \;\dee t.
\]
Here we abuse notation by writing $\delta(\ff,t) = \delta(\ff,I)$ for all $t\in I$.
\end{definition}

Definition \ref{definitiondimtemplate} can be understood intuitively in terms of a simple version of one-dimensional physics with sticky collisions and conservation of momentum. Suppose that we observe particles $P_1,\ldots,P_d$ travelling along trajectories $f_1,\ldots,f_d$ during a time interval $I$ along which $\ff$ is linear, and we want to infer the velocities of these particles before they collided, based on the following background information: before the collision $\pdim$ of the particles were travelling upwards at a speed of $\frac1\pdim$, and $\qdim$ of the particles were travelling downwards at a speed of $\frac1\qdim$. When particles collide (that is, when the velocities of the particles of lower index are more upwards than the velocities of the particles of higher index at the same location), they join forces to move as a unit, and their new velocity is determined by conservation of momentum. However, we can still think of the group as being composed of a certain number of ``upwards'' particles and a certain number of ``downwards'' particles.

The equations \eqref{Splusdef1} and \eqref{Sminusdef1} can be understood as suggesting a particular solution to this problem of inference: assume that within each group, all of the upwards-travelling particles started out below all of the downwards-travelling particles. This is not the only possible solution but it is the nicest one for certain purposes. Specifically, we can imagine a force of ``gravity'' attempting to bring all of the particles together, which acts between any two particles by imposing a fixed energy cost if the two particles are travelling away from each other.\Footnote{This is of course unlike real gravity, which imposes an energy cost variable with respect to distance.} The total energy cost is then the codimension $\dimprod - \delta(\ff,I)$ defined by \eqref{codimfI}. The equations \eqref{Splusdef1} and \eqref{Sminusdef1} can then be thought of as giving the solution that minimizes this cost.

The idea of codimension as an energy cost is also useful for computing the suprema \eqref{variational1} in certain circumstances, since it suggests principles like the conservation of energy. However, one needs to be careful since the stickiness of collisions means that some naive formulations of conservation of energy are violated.\\

In most cases of interest, the collection $\FF$ in Theorem \ref{theoremvariational1} is defined by some Diophantine condition. In this case, generally rather than $\MM(\FF)$ the set we are really interested in is the set of all matrices whose corresponding successive minima functions satisfy the same Diophantine condition. Although these two sets are \emph{a priori} different, Theorem \ref{theoremSSR}(i) implies that they are the same and thus Theorem \ref{theoremvariational1} is equivalent modulo Theorem \ref{theoremSSR}(i) to the following:

\begin{theorem}[Variational principle, version 2]
\label{theoremvariational2}
Let $\SS$ be a collection of functions from $\Rplus$ to $\R^d$ which is closed under finite perturbations, and let
\[
\MM(\SS) = \{\bfA : \Mink_\bfA \in \SS\}.
\]
Then
\begin{align}
\label{variational2}
\HD(\MM(\SS)) &= \sup_{\ff\in\SS\cap \TT_\dims} \underline\delta(\ff),&
\PD(\MM(\SS)) &= \sup_{\ff\in\SS\cap \TT_\dims} \overline\delta(\ff).
\end{align}
\end{theorem}
\begin{proof}[Proof of equivalence]
Theorem \ref{theoremvariational2} implies Theorem \ref{theoremvariational1} since we can take $\SS = \{\gg : \gg \asymp_\plus \ff \in \FF\}$. Conversely, Theorem \ref{theoremvariational1} implies Theorem \ref{theoremvariational2} modulo Theorem \ref{theoremSSR}(i) since we can take $\FF = \SS\cap\TT_\dims$.
\end{proof}

Theorem \ref{theoremvariational2} can be thought of as a quantitative analogue of Theorem \ref{theoremSSR}, as shown by the following equivalent formulation:

\begin{theorem}[Variational principle, version 3]
\label{theoremvariational3}
~
\begin{itemize}
\item[(i)] Let $S$ be a set of $\pdim\times\qdim$ matrices of Hausdorff (resp. packing) dimension $>\delta$. Then there exists a matrix $\bfA\in S$ and a template $\ff \asymp_\plus \Mink_\bfA$ whose lower (resp. upper) average contraction rate is $> \delta$.
\item[(ii)] Let $\ff$ be a template whose lower (resp. upper) average contraction rate is $> \delta$. Then there exists a set $S$ of $\pdim\times\qdim$ matrices of Hausdorff (resp. packing) dimension $>\delta$, such that $\Mink_\bfA \asymp_\plus \ff$ for all $\bfA\in S$.
\end{itemize}
\end{theorem}
\comdavid{In fact, Theorem \ref{theoremvariational4} will be the form of the variational principle that we will prove.}
\begin{proof}[Proof of equivalence]
Part (i) is equivalent to the $\leq$ direction of \eqref{variational2}, and part (ii) to the $\geq$ direction. For the first equivalence, for the forwards direction take $S = \{\bfA : \Mink_\bfA \in \SS\}$, and for the backwards direction take $\SS = \{\gg : \gg\asymp_\plus \Mink_\bfA, \; \bfA \in S\}$. For the second equivalence, for the backwards direction take $S = \MM(\ff)$ and $\SS = \{\gg : \gg\asymp_\plus \ff\}$.
\end{proof}

It is worth stating the special case of Theorem \ref{theoremvariational2} that occurs when the collection $\SS$ is defined by the Diophantine conditions defining $\Sing_\dims(\omega)$ and $\Sing_\dims^*(\omega)$ for some $\omega\geq \dirichlet$. Thus, we define the \emph{uniform dynamical exponent} of a map $\ff:\Rplus\to\R^d$ to be the number
\[
\what\dynexp(\ff) = \liminf_{t\to\infty} \frac{-1}t f_1(t).
\]
Similarly, $\ff$ is said to be \emph{trivially singular} if $f_{j+1}(t) - f_j(t) \to \infty$ as $t\to\infty$ for some $j = 1,\ldots,d-1$. Letting $\SS = \{\ff : \what\dynexp(\ff) = \dynexp\}$ or $\SS = \{\ff : \what\dynexp(\ff) = \dynexp, \; \ff \text{ not trivially singular}\}$ in Theorem \ref{theoremvariational2} yields the following result:

\begin{theorem}[Special case of variational principle]
\label{theoremvariational4}
For all $\omega \geq \dirichlet$, we have
\begin{align*}
\HD(\Sing_\dims(\omega)) &= \sup\{\underline\delta(\ff): \ff\in\TT_\dims,\;\;\what\dynexp(\ff) = \dynexp\}\\
\PD(\Sing_\dims(\omega)) &= \sup\{\overline\delta(\ff): \ff\in\TT_\dims,\;\;\what\dynexp(\ff) = \dynexp\}\\
\HD(\Sing_\dims^*(\omega)) &= \sup\{\underline\delta(\ff): \ff\in\TT_\dims,\;\;\what\dynexp(\ff) = \dynexp, \; \ff \text{ not trivially singular}\}\\
\PD(\Sing_\dims^*(\omega)) &= \sup\{\overline\delta(\ff): \ff\in\TT_\dims,\;\;\what\dynexp(\ff) = \dynexp, \; \ff \text{ not trivially singular}\},
\end{align*}
where $\dynexp$ is as in \eqref{dani}. 
\end{theorem}

Theorem \ref{theoremvariational2} can also be used to compute the dimensions of the set
\[
\w\Sing_\dims^*(\omega) = \{\bfA:\what\omega(\bfA) \geq \omega,\;\bfA\text{ not trivially singular}\} = \bigcup_{\omega'\geq \omega} \Sing_\dims^*(\omega').
\]
\begin{theorem}[Special case of variational principle]
\label{theoremvariational5}
For all $\omega \geq \dirichlet$, we have
\begin{align*}
\HD(\w\Sing_\dims^*(\omega)) &= \sup_{\omega' \geq \omega} \HD(\Sing_\dims^*(\omega'))\\
\PD(\w\Sing_\dims^*(\omega)) &= \sup_{\omega' \geq \omega} \PD(\Sing_\dims^*(\omega')).
\end{align*}
\end{theorem}
(Theorem \ref{theoremvariational5} is also true with the stars removed, but in that case it is not as interesting because $\HD(\Sing_\dims(\infty))$ is ``too large'', whereas $\HD(\Sing_\dims^*(\infty))$ is the ``correct'' size according to Remark \ref{remarktriviallysingular}.)

It is natural to expect that the map $\omega\mapsto \HD(\Sing_\dims^*(\omega))$ is monotonically decreasing, in which case Theorem \ref{theoremvariational5} would imply that
\[
\HD(\w\Sing^*(\omega)) = \HD(\Sing_\dims^*(\omega)).
\]

\begin{conjecture}
The functions
\begin{align*}
\omega &\mapsto \HD(\Sing_\dims^*(\omega)),&
\omega &\mapsto \PD(\Sing_\dims^*(\omega))
\end{align*}
are decreasing and continuous, and furthermore are computable in the sense of \cite{Weihrauch}.
\end{conjecture}

The main difficulty in proving this conjecture is the rigidity of templates -- one would like to show that every template can be perturbed into a new template whose uniform dynamical exponent is either slightly larger, or much smaller, than that of the original template, and whose average contraction rates are not too much smaller than those of the original template. However, it is not at all clear how one would perform such a perturbation except in a few special cases.

\section{A characterization of Hausdorff and packing dimensions using games}
\label{sectiongame}

The proof of the variational principle is based on a new variant of Schmidt's game which is in principle capable of computing the Hausdorff and packing dimensions of any set. In Schmidt's game \cite{Schmidt1}, players take turns choosing a descending sequence of balls and compete to determine whether or not the intersection point of these balls is in a certain target set. The key feature of our new variant is that instead of requiring the rate at which the players' moves contribute information to the game to be constant, the new variant allows the rate of information transfer to be variable, with the first player, Alice, getting to choose the rate of information transfer. However, Alice is penalized if she exerts too much control over the game over long periods of time without giving her opponent Bob a chance to exert control over the game.

\begin{definition}
\label{definitiongames1}
Given $0 < \beta < 1$, Alice and Bob play the \emph{$\delta$-dimensional Hausdorff (resp. packing) $\beta$-game} as follows:
\begin{itemize}
\item The turn order is alternating, with Alice playing first. Thus, Bob's $k$th turn occurs after Alice's $k$th turn and before Alice's $(k + 1)$st turn.
\item Alice begins by choosing a starting radius $\rho_0 > 0$.
\item On the $k$th turn, Alice chooses a nonempty $3\rho_k$-separated set $A_k \subset \R^d$, and Bob responds by choosing a ball $B_k = B(\xx_k,\rho_k)$, where $\xx_k \in A_k$ and $\rho_k = \beta^k \rho_0$.
\item On the first (0th) turn, Alice's choice $A_0$ can be any finite set, but on subsequent turns she must choose it to satisfy
\begin{equation}
\label{Alicerules}
A_{k + 1} \subset B(\xx_k,(1 - \beta)\rho_k).
\end{equation}
Note that this condition guarantees that
\[
B_0 \supset B_1 \supset B_2 \supset \cdots
\]
\end{itemize}
After infinitely many turns have passed, the point
\[
\xx_\infty = \lim_{k\to\infty} \xx_k \in \bigcap_{k = 0}^\infty B_k
\]
is computed (note that the right-hand side is always a singleton). It is called the \emph{outcome} of the game. Also, we let $\AA = (A_k)_{k\in\N}$, and we compute the numbers
\begin{equation}
\label{Hausdorff}
\underline\delta(\AA) \df \liminf_{k\to\infty} \frac{1}{k} \sum_{i = 0}^k \frac{\log\#(A_i)}{-\log(\beta)}
\end{equation}
and
\begin{equation}
\label{packing}
\overline\delta(\AA) \df \limsup_{k\to\infty} \frac{1}{k} \sum_{i = 0}^k \frac{\log\#(A_i)}{-\log(\beta)}
\end{equation}
which are called Alice's \emph{lower} and \emph{upper scores}, respectively. Alice's goal will be to ensure that the outcome is in a certain set $S$, called the \emph{target set}, and simultaneously to maximize her score while doing so.

To be precise, a set $S \subset \R^d$ is said to be \emph{$\delta$-dimensionally Hausdorff (resp. packing) $\beta$-winning} if Alice has a strategy to simultaneously ensure that the outcome $\xx_\infty$ is in $S$, and that her lower (resp. upper) score is at least $\delta$. It is said to be \emph{$\delta$-dimensionally Hausdorff (resp. packing) winning} if it is $\delta$-dimensionally Hausdorff (resp. packing) $\beta$-winning for all sufficiently small $\beta > 0$.
The reader may contrast the definition of our game with that of Cheung's {\it self-similar coverings} in \cite[Section 3]{Cheung}.
\end{definition}

The following result is one of the key ingredients in the proof of the variational principle:

\begin{theorem}
\label{theoremHPgame}
The Hausdorff (resp. packing) dimension of a Borel set $S$ is the supremum of $\delta$ such that $S$ is $\delta$-dimensionally Hausdorff (resp. packing) winning.
\end{theorem}

\bibliographystyle{amsplain}

\bibliography{bibliography}

\end{document}